\documentclass{amsart}
\usepackage{cleveref}
\usepackage{amssymb}
\usepackage{fullpage}
 \usepackage{color}

 \renewcommand{\ge}{\geqslant}
 \renewcommand{\le}{\leqslant}

  \newcommand{\x}{\mathbf{x}}
  \newcommand{\z}{\mathbf{z}}
  
  \newcommand{\0}{\mathbf{0}}

 \newcommand\bb{{\mathbf b}}

 \newcommand\bg{{\mathbf g}}
 \newcommand\bi{{\mathbf i}}
 \newcommand\bj{{\mathbf j}}

 \newcommand\bV{{\mathbf V}}

 \newcommand\cS{{\mathcal S}}
 \newcommand\cT{{\mathcal T}}

 \newcommand\rP{{\rm P}}
 
 \newcommand\rS{{\rm S}}

  \newcommand{\n}{\mathbf{n}}

  \newcommand{\trans}{^\top}

\newcommand{\C}{\mathbb{C}}
\newcommand{\F}{\mathbb{F}}
\newcommand{\G}{\mathbb{G}}

\newcommand{\N}{\mathbb{N}}
\newcommand{\Q}{\mathbb{Q}}

\newcommand{\Z}{\mathbb{Z}}

\newtheorem{theorem}{Theorem}
\newtheorem{proposition}[theorem]{Proposition}
\newtheorem{lemma}[theorem]{Lemma}

\theoremstyle{definition}

\theoremstyle{remark}

\newcommand{\rank}{\operatorname{rank}}
\newcommand{\srank}{\operatorname{srank}}
\newcommand\blfootnote[1]{%
  \begingroup
  \renewcommand\thefootnote{}\footnote{#1}%
  \addtocounter{footnote}{-1}%
  \endgroup
}

\begin{document}
\title{On the complexity of finding tensor ranks}
\author[1]{Mohsen Aliabadi \textsuperscript{*} \and Shmuel Friedland \textsuperscript{**}}

\blfootnote{$^{*}${Department of Mathematics, Iowa State University, Carver Hall, 411 Morrill Rd, Ames, IA 50011, USA, aliabadi@iastate.edu}}
\blfootnote{$^{**}${Corresponding author, Department of Mathematics, Statistics and Computer Science,  University of Illinois, 851 South Morgan Street, Chicago, Illinois 60607-7045, USA, friedlan@uic.edu}}


\begin{abstract}
The purpose of this note is to give a linear algebra algorithm to find out if a rank of a given tensor over a field $\F$ is at most $k$ over the algebraic closure of $\F$, where $k$ is a given positive integer.
We estimate the arithmetic complexity of our algorithm.
\end{abstract}
\maketitle
 

\section{Introduction}\label{sec:intro}
In the last fifty years it became clear that multiarrays with more than two indices, known as tensors, are vital tools in data processing \cite{LC18}, mathematical biology \cite{BAPHA}, numerical linear algebra \cite{Str87}, quantum physics \cite{BFZ}, theoretical computer science \cite{BCS97} and theoretical mathematics \cite{La12}.
Formally, we define a $d$-tensor, as a multiarray with $d\ge 3$ indices. 
One of the simple criterion of the complexity of a given tensor is its rank. 
(The concept of tensor rank has been introduced in the early 20th century \cite{Hit27}.)
Recall that the rank of a nonzero tensor is the minimum number of terms in a decomposition of tensor as a sum of rank-one tensors.
The famous Strassen algorithm for multiplication of two matrices of order two using seven multiplication \cite{Str69} and not less, is equivalent to the statement that the corresponding $4 \times 4\times 4$ tensor has rank seven.  The 
3-satisfiability problem with $n$ variables and $m$ clauses can be stated if a given $3$-tensor has a specific rank \cite{Has90}.  This result yields that the computation of the rank of tensor over any finite field is NP-complete, and is NP-hard over fields of rational, real and complex numbers.  
 
On the other hand, the rank of a matrix is a well-understood notion, which has many equivalent  definitions.  The computation of the rank of matrix is usually obtained by applying the Gaussian elimination process: Namely, it is the number of non-zero rows in the row echelon form obtained from the Gaussian elimination process.
See \cite{FA18} for classical results on matrix rank.

The point of this note is the following statement.  Suppose that $\cT$ is a $d$-tensor
over a given field $\F$.   Denote its rank  by   rank$_\F \cT$.  Assume that $\G$ is an extension field of $\F$.  Then rank$_\G \cT\le$rank$_\F\cT$, and strict inequality may hold \cite{FL18}.  Denote by $\hat \F$ the algebraic closure of $\F$.  
Let $k\ge 2$ be an integer.  Then rank$_{\hat \F}\cT>k$ if and only if certain systems of linear equations are solvable over $\F$.  Since rank$_{\hat F}\cT$ is NP-hard to compute one expects the linear system is exponential in number of variables.  
Our techniques also apply to symmetric tensors and their symmetric rank \cite{CGLM}.
Our main result is a consequence of an effective Nullstellensatz \cite{Kol}.

The main drawback of our approach is a huge number of variables and equations that one encounters in trying to apply an effective Hilbert Nullstellensatz. We hope that our approach can be further improved to a smaller number of variables and equations for the specific problem of tensor rank. 

Over finite fields, one can study directly the problem of determining  if the rank of a given tensor is at most $k$.  A recent paper \cite{LPTWY} gives a probabilistic algorithm to find solution of polynomial equations over a finite field with high probability.  
We will compare the complexity of our method with the above result.

\section{Preliminary results}\label{sec:prel}
Let $\F$ be a field and $d\in\N$.   Denote: 
\begin{eqnarray*}
[d]=\{1,\ldots,d\},\n=(n_1,\ldots,n_d)\in\N^d,  [\n]=[n_1]\times \cdots\times [n_d], N(\n)=\prod_{j=1}^d n_j, L(\n)=\sum_{j=1}^d n_i, \F^{\n}=\otimes_{j=1}^d \F^{n_j}.  
\end{eqnarray*}
The entries of $\cT\in\F^{\n}$ are denoted by $\cT_{i_1,\ldots,i_d}, i_j\in [n_j],j\in[d]$, and  we view $\cT$ as $[\cT_{i_1,\ldots,i_d}]$.  Then $\cT$ is a matrix for $d=2$ and  a tensor for $d\ge 3$.  A tensor $\cT\in\F^{\n}\setminus\{0\}$ is called a  {rank-one} tensor if $\cT=\otimes_{j=1}^d \x_j$ for $\x_j\in\F^{n_j}\setminus\{\0\}, j\in[d]$.  Recall that for a tensor $\cT\in\F^{\n}\setminus\{0\}$ the rank of 
$\cT$ is the minimal number of terms in the decomposition $\cT=\sum_{i=1}^r \otimes_{j=1}^d \x_{j,i}$.  (The rank of zero tensor is zero.)

The unfolded rank of 
$\cT\in\F^{\n}$ in mode $j\in[d]$  denoted as $r_j(\cT)$, is defined as follows:  For simplicity of exposition, let us explain the notion $r_1(\cT)$.  View tensor $\cT$ as a matrix  $T\in\F^{n_1}\otimes(\otimes_{j=2}^d \F^{n_j})$.  Then $r_1(\cT)$ is $\rank T$.  We denote by $\bg_{1,1},\ldots,\bg_{r_1(\cT),1}$  a column basis of $T$.   Clearly, $r_1(\cT)\le n_1$.  Let $\bV_1=$\ span$(\bg_{1,1},\ldots,\bg_{r_1(\cT),1})\subseteq \F^{n_1}$.   Similarly we define the rank $r_j(\cT)$, the column space $\bV_j\subseteq \F^{n_j}$ and a basis $\bg_{1,j},\ldots,\bg_{r_j(\cT),j}$ in $\bV_j$.  It is known that for $d\ge 3$ it is possible that all $r_j(\cT), j\in[d]$ are different.
Observe that $\cT$ can be viewed as a tensor $\otimes_{j=1}^d \bV_j$.  Introduce a new basis in $\F^{n_j}$, such that a basis of $\bV_j$ is part of this basis. Hence we can convert the tensor $\cT$  {to} a ``smaller'' tensor $\cT'\in \F^{\mathbf{r}}, \mathbf{r}=(r_1(\cT),\ldots,r_d(\cT))$.    
For simplicity of  the exposition we assume
\begin{eqnarray}\label{normal1}
1\le n_1\le \cdots\le n_d, \; r_j(\cT)=n_j, \textrm{ for }j\in[d].
\end{eqnarray}
It is well-known that $n_d\le \rank \cT\le \prod_{i=1}^{d-1}n_i$ \cite{Fri12}.
Thus we are going to assume  {that}
\begin{eqnarray}\label{norm2}
n_d\le r\le \prod_{i=1}^{d-1}n_i.
\end{eqnarray}
Clearly, $\rank\cT>r$ if and only if the system 
\begin{eqnarray}\label{poleqgenten1}
\sum_{i=1}^r \otimes_{j=1}^d \x_{j,i}-\cT=0
\end{eqnarray}
is not solvable over $\F$.

Assume that $n_1=\cdots=n_d=n$.  Denote $n^{\times d}:=\n$.  A tensor $\cS\in \F^{n^{\times d}}$ is called symmetric if $\cS_{i_1,\ldots,i_d}=\cS_{i_{\sigma(1)},\ldots,i_{\sigma(d)}}$ for each $i_1,\ldots,i_d\in [n]$ and each bijection $\sigma:[d]\to[d]$.  We denote by $\rS^d\F^n\subset \F^{n^{\times d}}$ the subspace of symmetric tensors. It turns out  that $\dim\rS^d\F^n={n+d-1\choose d}$.  As $\rS^d\F^n\subset \F^{n^{\times d}}$ it follows that ${n+d-1\choose d}\le n^d$.
It is  {well}-known that a symmetric tensor has rank one if $\cS=a\otimes^d\x$, where $a\in\F\setminus\{0\}, \x\in\F^n\setminus\{\0\}$.  Also if $\F$ has at least $d$ elements then each $\cS\in\rS^d\F^n$ is a sum of  {rank-one} symmetric tensors \cite[Proposition 7.2]{FS13}.  (It is shown in \cite[Proposition 7.1]{FS13} that for a fixed finite field $\F$ and $n\ge 2$ there exist  symmetric tensors which are not sum of  {rank-one} symmetric tensors for  sufficiently large $d$.)  In  {the following passage}, we assume that $|\F|\ge d$.   {We define $\srank \cS,$ the symmetric rank of $\cS\in \rS^d\F^n\setminus\{0\}$, as the minimal number in the decomposition of $\cS$ as a sum of {rank-one} symmetric tensors.}  For matrices over a field of characteristic $\ne 2$ the symmetric rank of $\cS\in\rS^2\F^n$ is equal to the (standard) rank of $\cS,$  {whereas for} $d\ge 3$ there are examples of $3$-symmetric tensors whose symmetric rank is greater than their tensor rank \cite{Shi17}.
Observe that for a symmetric $\cS\in \rS^d\F^n$ one has the equality $r_j(\cS)=r(\cS)$ for each $j\in[d]$.  In this case we assume that $n=r_1(\cS)$ and $n\le r$.
\\
\indent
Let $f(\x)$ be a homogeneous polynomial of degree $d$ in $n$ variables:
\begin{eqnarray*}
&&f(\x)=\sum_{ j_k+1\in [d+1],k\in[n], j_1+\cdots +j_n=d} \frac{d!}{j_1!\cdots j_n!} f_{j_1,\ldots,j_n} x_1^{j_1}\cdots x_n^{j_n}=\\
&&\sum_{\bj\in J(d,n)} c(\bj)f_{\bj}\x^{\bj}, \quad
\x^{\bj}=x_1^{j_1}\cdots x_n^{j_n}, c(\bj)=\frac{d!}{j_1!\cdots j_n!},f_{\bj}=f_{j_1,\ldots,j_n},\\
&&J(d,n)= \{\bj=(j_1,\ldots,j_n)\in\Z_+^n,\; j_1+\cdots+j_n=d\}.
\end{eqnarray*}
Conversely, a homogeneous polynomial $f(\x)$ of degree $d$ in $n$ variables defines a unique symmetric $\cS\in\rS^d\F^n$, which is given by the scalar product of the symmetric tensors $\cS$ and $\otimes^d \x$.  (See part (4) of Lemma 1 in\cite{FW18}.)   Denote by $\rP(d,n,\F)$ the space of all homogeneous polynomials of degree $d$ in $n$ variables over $\F$.

A polynomial in $n$ variables of degree at most $d$ has the representation
\begin{equation}\label{defpolfx}
f(\x)=\sum_{ j_k+1\in [d+1],k\in[n], j_1+\cdots +j_n\le d} \frac{d!}{j_1!\cdots j_n!} f_{j_1,\ldots,j_n} x_1^{j_1}\cdots x_n^{j_n}.
\end{equation}
Let $\x'=(x_1,\ldots,x_{n+1})=(\x,x_{n+1})$.  Then $f(\x)$ of the above form is $g(\x,1)$ for a unique $g(\x')\in \rP(d,n+1,\F)$.

Assume that $\G$ is an extension field of $\F$.  Given $\cT\in\F^{\n}\setminus{0}$, one can ask what  the rank of $\cT$ over $\G$ is?  That is,  what is the minimum number of terms in a decomposition of $\cT$ as a sum of rank-one tensor, where each rank-one tensor is in $\G^{\n}$.   We denote this rank by rank$_\G\cT$.  (When no ambiguity arises we denote rank$_\F\cT$ by $\rank\cT$.)  Clearly, rank$_\G\cT\le$rank$_\F\cT$.  It is well-known that in some cases strict inequality holds \cite{FL18}.   Similar results hold for symmetric rank of symmetric tensors.
\section{Outline of our approach}\label{sec:approach}
It is well-understood that matrices are closely related to linear transformations, while tensors are closely related to polynomial maps \cite{Fri12}.   More precisely,
as  {we have mentioned,} it is a classical result that the rank of a matrix is polynomially computable, using the Gauss elimination.  
 {In contrast to matrix rank, the tensor rank is unfortunately NP-hard to compute, as proven by Hastad \cite{Has90}.} Later Schaefer and Stefankovic \cite{SS18}  {showed} that  { determining the rank of a tensor over a field has the same complexity as deciding the existential theory of the field, which implies Hastad's NP-hardness results.}  Another result of Hastad states that the rank  decomposition problem is NP-complete  in the case of finite fields. Recently, Shitov \cite{Shi16} showed that rank over a field $\F$ is complete for the existential theory of $\F$ and is also uncomputable over $\Z$. 
 
{Let $\F$ be a given field and $\hat\F$ be its algebraic closure. Let $d\ge 3.$ We aim to find the upper bounds for the bit complexities for the following problems: 
\begin{enumerate}
\item For a tensor $\cT\in \F^{\n},$ to determine if $\rank \cT$ over the field $\hat \F$ is $\le r,$ for a fixed integer $r \ge 2;$
\item For a symmetric tensor $\cS\in\rS^d\F^n$, to determine if $\srank \cS$ over $\hat \F$ is $\le r,$ for a fixed integer $r \ge 2$.
\end{enumerate}
}

We start with the following obvious lemma:
\begin{lemma}\label{poleqgenten}  Let $3 \le d$ and $2\le n_1\le \ldots\le n_d$ be integers.  Assume that $\cT\in \F^{\n}$.  Then $r<\rank \cT$ if and only if the system of 
$N(\n)$ polynomial equations \eqref{poleqgenten1}
in $rd$ vector variables $\x_{1,1},\ldots, \x_{d,1},\ldots,\x_{1,r},\ldots,\x_{d,r}$, with a total number of $rL(\n)$ variables, of degree $d$ is not solvable. 
\end{lemma}

Note  that to decide the rank of $3$-mode tensor $\cT$ over $\C$ is an NP-hard problem  while deciding the rank of $3$-mode tensor over $\F=\Z/(p\Z)$ is an NP-complete problem. 

Over an algebraically closed field, this statement is equivalent to the fact that the ideal generated by $N(\n)$ polynomials which are entries of the left-hand side of \eqref{poleqgenten1} contains the constant function $1$. 
Thus for a given $r\ge n_d$ the system \eqref{poleqgenten1} reduces to $N(\n)$ equations, with $M(r)=rL(\n)$ variables.  Note that $M(r,\n)\le dr^2$ variables. For simplicity of the exposition we are going to assume that $r$ is small enough so that the number of variables is  less than or equal to  the number of equations:
\begin{eqnarray*}
M(r,\n)=rL(\n)\le N(\n).
\end{eqnarray*}
We now recall an efficient version of Hilbert Nullstellensatz \cite[1.9. Corollary]{Kol}. 
(See \cite{Bro87,Bro,EL99} for recent improvements on Hilbert's Nullstellensatz.)
 Assume first that $\F=\C$.  Denote by $\Z[\bi]$ the Gaussian integers.
 
\begin{lemma}\label{solvlem}
 Let $\cT\in\Z[\bi]^{\n}$ be given.  Then the complexity of deciding if $\rank \cT> r$ is at most the complexity  of finding if the linear system in the coefficients of polynomials $g_1,\ldots,g_{N(\n)}$
\begin{eqnarray}\label{Bezeq}
&&\sum_{i=1}^{N(\n)}g_i(\z)f_i(\z)=1,\\ 
&&g_i(\z)\in \C[\C^{M(r,\n)}], \; \deg g_i\le d^{M(r,\n)-1}, \;i\in [N(\n)],
\notag
\end{eqnarray} 
is solvable over $\Q[\bi]$ in precise arithmetic.   
\end{lemma}
\begin{proof}  Lemma \ref{poleqgenten} yields that the system of 
$N(\n)$ polynomial equations \eqref{poleqgenten1} is not solvable if and only if $\rank \cT>r$.  As $\C$ is algebraically closed the assumption that $\rank \cT>r$ yields that there exist $N(\n)$ polynomials $g_1,\ldots,g_{N(\n)}$ such that the system \eqref{Bezeq} is solvable.   The efficient version of Hilbert Nullstellensatz \cite[1.9. Corollary]{Kol} yields that the degree of each $g_i$ is at most $d^{M(r,\n)-1}$.
Write each $g_i(\z)$ in the form \eqref{defpolfx} of degree $d^{M(r,\n)-1}$, where the monomial coefficients of each $g_i(\z)$ are unknown variables.  Then the existence of $g_1,\ldots, g_{N(\n)}$ of degrees at most $d^{M(r,\n)-1}$ that satisfy \eqref{Bezeq}  is equivalent to the solvability of the system of linear equations in the monomial coefficients of each $g_i(\z)$ induced by \eqref{Bezeq}.
\end{proof}

 Thus for a fixed $r$ the complexity of determining the solvability of this system of linear equations is as follows. 
 
One can view a polynomial $p(\z)$ of degree $ d^{M(r,\n)-1}$ in $M(r,\n)$ variables as a homogeneous polynomial of degree  $ d^{M(r,\n)-1}$ with $M(r,\n)+1$ variables, where the variable $x_0$ has value $1$.  Hence the number of monomials appearing in $p(\z)$ is ${M(r,\n)+d^{M(r,\n)-1}\choose d^{M(r,\n)-1}}={M(r,\n)+d^{M(r,\n)-1}\choose M(r,\n)} $.
As we observed, 
\begin{eqnarray*}
&&{M(r,\n)+d^{M(r,\n)-1}\choose M(r,\n)} = {(d^{(M(r,\n)-1}+1)+M(r,\n)-1\choose M(r,\n)}\le (d^{(M(r,\n)-1}+1)^{M(r,\n)}=\\
&&(1+d^{-(M(r,\n)-1)})^{M(r,\n)} d^{(M(r,\n)-1)M(r,\n)}\le(1+1/M(r,\n))^{M(r,\n)}d^{(M(r,\n)-1)M(r,\n)}\le ed^{(M(r,\n)-1)M(r,\n)}.
\end{eqnarray*}
Here $e=\lim_{m\to\infty}(1+1/m)^m=2.718\ldots$.
Hence the total number of coefficients of monomials in each 
$g_i(z)$ is bounded above by $ed^{(M(r,\n)-1)M(r,\n)}$.  We call these coefficients linear variables.
 Thus the total number of linear variables is bounded above is $e N(\n)d^{M(r,\n)(M(r,\n)-1)}$.  The number of equations is the number of monomials which is bounded above by $ed^{M(r,\n)(M(r,\n)-1)}$.   Thus if we use Gauss elimination to determine if this system of linear equations is solvable or not we need $O(N(\n)d^{3M(r,\n)(M(r,n)-1)})$ flops.  Ignoring the factor $N(\n)$, we will need  $O(d^{3M(r,\n)(M(r,n)-1)}))$ flops.  To estimate the computational complexity we also need to take into account the storage space in terms of the entries of $\cT$.  This is done in the next section.
 
It seems that in some cases it would be beneficial to reduce the number of variables as follows.  Note that the number  of variables for rank-one tensor can $\x_1\otimes \cdots\otimes \x_d$ is $-d+1+\sum_{i=1}^d n_i$. Indeed we can always assume that for $i<d$ one of the coordinates of $\x_i$ is $1$.  However we do not know which coordinate is $1$.  So we need to choose the place of this coordinate.  There are $n_i$ choices.  Hence we need to take $N(\n')=n_1\cdots n_{d-1}$ choices  for each rank-one tensor.  (Here $\n'=(n_1,\ldots, n_{d-1})$.)  Thus for rank $r$ we have $N(\n')^r$ choices to consider.  Hence we can replace our complexity estimate $O(d^{3M(r,\n)(M(r,n)-1)})$ by $O(N(\n')^rd^{3(M(r,\n)-d+1)(M(r,n)-d)})$.

Assume that $\G$ is algebraically closed.  A rank-one symmetric tensor in $\rS^d\G^n$
is of the form $\x^{\otimes d}=\x\otimes\cdots\otimes\x, \x\in\G^n$.  A Waring decomposition of 
$\cS\in \rS^d\G^n\setminus\{0\}$ is $\cS=\sum_{i=1}^r \x_i^{\otimes d}$ \cite{aa}.   Since every algebraically closed field has an infinite number of elements it follows that every symmetric tensor has a Waring decomposition  \cite{FS13}.  The minimal number of rank-one symmetric tensors in the decomposition of $\cS$  is called a symmetric rank and is denoted as $\srank\cS$.  Clearly, $\rank\cS\le \srank \cS$.   It is shown in \cite{Fri16,ZHQ16} that in certain cases one has equality $\rank\cS=\srank\cS$.  However, even for $d=3$ one can have an inequality $\rank\cS<\srank\cS$ \cite{Shi17}.

Hence an analog of Lemma \ref{poleqgenten} is:
\begin{lemma}\label{poleqsymten}  Let $3 \le d$ and $2\le n$ be integers.  Assume that $\cS\in \G^{\n}$, where $\G$ is an algebraically closed field.  Then $r<\srank \cS$ if and only if the following system of 
$n+d-1\choose d$ polynomial equations 
\begin{eqnarray}\label{poleqsymten1}
\sum_{i=1}^r \x_i^{\otimes d}-\cS=0
\end{eqnarray}
is not solvable.
\end{lemma}

Assume that $\cS\in\rS^d\F^{n}$.
Thus $\srank \cS>r$ over $\hat \F$ if the above system is not solvable over $\hat \F$.
Let $f_{i_1,\ldots,i_d}(\x_1,\ldots,\x_r)$ be the left hand side of $(i_1,\ldots,i_d)\in[n]^d$ entry.  Since we are dealing with symmetric tensors we can assume that $1\le i_1\le \cdots\le i_d\le n$.  Hence the  {unsolvability} of \eqref{poleqsymten1} is equivalent to
\begin{eqnarray}\label{Bezeq1}
\sum_{(1\le i_1\le\cdots\le i_d\le n} g_{i_1,\ldots,i_d}(\x_1,\ldots,\x_r)f_{i_1,\ldots,i_d}(\x_1,\ldots,\x_r)=1.
\end{eqnarray}
Note that the total number of variables is $rn$.  The efficient Nullstellensatz \cite{Kol} gives an upper bound on $\deg g_{i_1,\ldots,i_d}f_{i_1,\ldots,i_d}\le d^{rn}$.  The arguments above yield that the number of monomials of degree at most $d_i=\deg  g_{i_1,\ldots,i_d}$ is less than $ed^{(rn)(rn-1)}$.  Viewing the coefficients of $g_{i_1,\ldots,i_d}$ as linear variables we deduce that the total number of linear variables in all $g_{i_1,\ldots,i_d}$ is at most ${n+d-1\choose  d}d^{(rn)(rn-1)}$.  In \S\ref{sec:prel} we showed that ${n+d-1\choose  d}\le n^d$.  Observe next 
\begin{eqnarray*}
{n+d-1\choose  d}={n+d-1\choose  n-1}={d+1+(n-1)-1\choose  n-1}\le (d+1)^{n-1}.
\end{eqnarray*}
Hence the number of linear variables is bounded above by $O(\min(n^d,(d+1)^{n-1})d^{(nr)(nr-1)})$.
The number of equations is as the number of monomials which is bounded above by $d^{(nr)(nr-1)}$.
Thus if we use Gauss elimination to determine if this system of linear equations is solvable or not, we need $O({n+d-1\choose  d}d^{3(nr)(nr-1)})$ flops.

 A complementary question is, suppose that we know that rank$_{\hat F}\cT=r$ using the above approach.   Does  the complexity of finding its rank decomposition in some explicit way  have  roughly the same complexity as finding that rank$_{\hat F}\cT=r$?
 This is a much harder problem discussed   in \cite{FW18}.
 In \cite{Nie17} Nie gives an algorithm on  finding  solvability of such a decomposition for symmetric tensor over $\C$.  No complexity analysis is investigated though.

\section{Complexity of solvability of linear systems over integers}
We provide a  simple complexity result on the solvability   of nonhomogeneous linear system of equations 
\begin{equation}\label{syslineq}
A\x=\bb, \quad A\in\Q^{m\times n},\,\bb\in \Q^m\setminus\{\0\}.  
\end{equation}
We represent each rational number by $p/q$, where $p\in\Z$ and $q\in \N$.
We do not assume that  $p,q$ are coprime.  The storage for $p/q$, can be also written as $(p,q)$ is 
$$h(p/q)=\lceil \log_2 q\rceil +\max(1,\lceil\log_2(2|p|)\rceil).$$
We denote by $H$ the maximum  {height} of the augmented matrix $\hat A=[A\,\bb]$.

The system is solvable if and only if one does not have pivots in the last column of $\hat A$.  Equivalently,    the Kronecker-Capelli theorem  claims that the system is solvable if and only if $\rank A=\rank \hat A$.
 The number of operations that one needs  is $O(mn^2)$.  However if $n> m$, then one can compute $\rank A$ and $\rank \hat A$ using the equalities $\rank A=\rank A\trans$ and $\rank \hat A=\rank \hat A\trans$.  This will give the number of operations  $O(nm^2)$.  
Hence the number of flops we need to carry out is $O(\min(mn^2,nm^2))$. 
We need to address the storage of the entries, whose  $h$ function is growing when we  perform the Gauss elimination.  

Recall the complexity of computation of the product of two positive integers $p$ and $q$.  The standard algorithm would take $O(h(p)h(q))$.   However, there are better algorithms: Karatsuba algorithm, Toom-Cook multiplication algorithm and Schonhage-Strassen algorithm.  Basically, if one assumes the Schonhage-Strassen algorithm, it follows that the number of operations for the product of two numbers of height at most $H$ is $O(H\log H\log\log H)$.  For simplicity of notation we will ignore the logarithmic factors by denoting the complexity of the Schonhage-Strassen algorithm as $O(H)$.

  \begin{proposition}\label{cmplexsolvlinsys}  Consider the system of linear equations \eqref{syslineq}.  Then the complexity of determining the solvability of this system is   $O(\max(m,n)\min(m,n)^4H) $, where the logarithmic terms taking into account  the Schonhage-Strassen algorithm are suppressed.  If $A$ and $\bb$ are integers then the complexity of determining the solvability of this system is   $O(\max(m,n)\min(m,n)^3(H+\log_2\min(m,n))) $.
 \end{proposition}
 \begin{proof}
 We first assume that $A$ and $\bb$ have integer entries.
 We compute ranks of $A$ and $\hat A$ and use the Kronecker-Capelli theorem.
 By considering $A\trans, \hat A\trans$, if needed, we will assume that $n\ge m$.
 We perform the standard Gauss elimination on $A$, without normalizing the pivots, as in \cite{Edm67},  \cite[\S3.3]{Sch} or \cite[\S1.3.2]{FA18}.
 Let $D_k$ be the determinant of the $k\times k$  submatrix of $A$ that contains the $k$ pivots.  Then $|D_k|\le 2^{kH} k!$.  Hence $\log_2 |D_k|=O(k(H+\log_2 k))$.
The main observation is that the value of $k$-th pivot is the ratio of the corresponding $D_k/D_{k-1}$, where $D_0=1$.   Hence the height of $k$-th pivot is $O(k(H+\log k))$.
The Schonhage-Strassen algorithm for multiplying two integers by this height is $O(k(H+\log _2k))$.  As $k\le m$, we get that the storage of all entries is $O(nm^2(H+\log_2 m))$.  As we need $O(nm^2)$ flops we deduce that the total complexity is $O(nm^3(H+\log_2 m))$.

Assume now that $A$ and $\bb$ have rational entries. Then we multiply each nonzero column of $A$ by the product of the numerators of the entries of this column to obtain $A_1\in \Z^{m\times n}$.  Similarly we obtain $\bb_1\in \Z^m$.  Clearly $\rank A_1=\rank A$ and $\rank \hat A_1=\rank \hat A$.  Observe that $H_1$, the height of $\hat A_1$ is $O(mH)$.  Hence the complexity of determining solvability of \eqref{syslineq} is $O(nm^3(mH+\log_2 m))=O(nm^4 H)$.
\end{proof}
\section{Polynomial equations of finite field}
 Let $\F_q$ be a finite field with $q=p^l$ elements, where $p$ is a prime number and $l$ a positive integer.  Assume that $\cT\in \F_q^{\n}$.  
 Then $r\ge \rank \cT$ over $\F_q$ if and only if the system of polynomial equations \eqref{poleqgenten1} is solvable.  A brute force method by checking all possible values of $M(r,\n)$ variables is $O(q^{M(r,\n)})$.
 In a recent paper \cite{LPTWY} a somewhat better result is given.  Namely, there is a randomized  algorithm of running time $O(q^{(1-\delta(r,p,l))M(r,\n)})$ with high probability finding a solution of the system if it is solvable. For $q=2$ one has that $\delta(r,2,1)\approx 0.1135$.  For $q\le 2^{4erl}$ one has that $\delta(r,p,l)=O(1/r)$.
Otherwise, $\delta(r,p,l)=\frac{l\log(\log(q/4erl))}{\log q}$.  
 
 Note that for our method we need to determine the solvability of the system of linear equations over $\F_q$.  Thus the storage needed for each entry is $O(\log_2 q)$.  The complexity of multiplying using two elements in $\F_q$ is $O(q)$ ignoring the logarithmic factors.
 Hence the complexity for determining the solvability of system of $m$ equations in $n$ unknowns using Gauss elimination is $O(\max(m,n)\min(n,m)^2 q)$.
 Thus for $q>d^{3M(r,\n)}$
 our complexity for finding if rank$_{\hat\F}\cT>r$ is comparable or better than the complexity of finding if rank$_{\F_q}\cT>r$.  
 
 Similar results hold for symmetric tensors.  Assume that $\cS\in\rS^d\F_q^n$. Then srank$_{\F_q}\cS>r$ if and only if the following system of ${n+d-1\choose d}$ polynomial equations is not solvable over $\F_q$: 
 \begin{eqnarray}\label{poleqsymten2}
\sum_{i=1}^r t_i\x_i^{\otimes d}-\cS=0,
\end{eqnarray}
where $t_i\in\F_q, \x_i\in\F_q^n$ for $i\in[r]$.  Note that the total number of variables  is $(r+1)n$.

\section{Open problems}
The first major problem is  {whether we can} reduce the number of monomials appearing in $g_{i_1,\ldots,i_d}$.  The insight behind this problem is that each $f_{i_1,\ldots,i_d}$ consists of a constant term and $-\cT_{i_1,\ldots,i_d}$ and sum of $r$ multilinear monomials which are invariant under the permutation of the $r$ vectors $\{\x_{j,1},\ldots,\x_{j,r}\}\to \{\x_{j,\sigma(1)},\ldots,\x_{j,\sigma(r)}\}$ for $j\in[d]$, and $\sigma:[d]\to[d]$.  The second problem; is it true that  each monomial of $g_{i_1,\ldots,i_d}$ is a monomial in the entries of rank-one tensors $\otimes_{j=1}^d x_{j,i}$? Further investigation along those lines could prove to be worthwhile.

\section*{Acknowledgment}
The work of the second author is partially supported by the Simons Collaboration Grant for Mathematicians.

 \end{document}